\documentclass[12pt]{amsart}

\usepackage{amsmath,amsfonts,amsthm,amssymb,amscd}
\usepackage{epsfig}
\usepackage{inputenc}
\usepackage{upref}
\usepackage{bm}
\usepackage{calc}
\usepackage{caption}
\usepackage{subcaption}
\usepackage{hyperref}
\usepackage[usenames,dvipsnames]{xcolor}
\usepackage[normalem]{ulem}
\usepackage{wasysym}
\usepackage{mathrsfs}
\usepackage{enumerate}
\usepackage{setspace}
\usepackage{array}
\usepackage{cancel}

\newtheorem{theorem}{Theorem}[section]
\newtheorem*{theorem*}{Theorem}

\newtheorem{lemma}[theorem]{Lemma}

\newtheorem{remark}[theorem]{Remark}

\begin{document}
\title{${l}^2$ Decoupling for certain degenerate surfaces in $\mathbb{R}^4$}
\author{Kalachand Shuin }
\address{Kalachand Shuin\\
	Department of Mathematics\\
	Indian Institute of Science\\
	Bengaluru, 560012, India.}
\email{shuin.k.c@gmail.com}

	\maketitle
 \textbf{Abstract:} In this article, we aim to study decoupling inequality for a specific degenerate hypersurface in $\mathbb{R}^4$. Inspired by the work of Bourgain--Demeter and Li--Zheng,  we consider the hypersurface $\mathcal{S}^3_{4}:=\{(\xi_1,\xi_2,\xi_3,\xi^4_1+\xi^4_2+\xi^4_3):0\leq \xi_j\leq1, \text{for}~j=1,2,3\}$ in $\mathbb{R}^4$ and study decoupling estimates.

\section{Introduction}
Decoupling inequality, a novel concept in Euclidean harmonic analysis, manifests as a variant of almost orthogonality in the non-Hilbert spaces. It's inception traces back to Wolff's exploration in his local smoothing work, as documented in \cite{Wolff}. Subsequently, in 2015, Bourgain and Demeter \cite{BD} proved the sharp decoupling conjecture for any compact $C^2$ hypersurface with positive definite secound fundamental form. Since then, numerous mathematicians have delved into the decoupling problem.
Decoupling inequality has numerous applications in the field of number theory, distance problems and PDEs \cite{DORZ, BDG, Guth}. Very recently, Gan and Wu \cite{GanWu} have  established weighted decoupling inequality and using that idea they have improved  the maximal Bochner--Riesz conjecture in every dimensions. In \cite{BD2016} Bourgain and Demeter proved decoupling theory for general nondegenerate surfaces in $\mathbb{R}^4$. Also see \cite{Oh, GOZK} for more about decoupling inequalities associated with quadratic forms and $3$-dimensional nondegenerate surfaces in $\mathbb{R}^6$. On the other hand, Biswas et. al. \cite{Biswas} have investigated decoupling inequality for finite-type curves in $\mathbb{R}^2$. Recently,  Li and Zheng \cite{LiZheng} have studied decoupling inequality of degenerate hypersurfaces in $ \mathbb{R}^3$. The result of \cite{LiZheng} can be deduced from the work of  Li and Yang \cite{Yang}. In \cite{Yang}, the authors have investigated decoupling for mixed-homogeneous polynomials in $\mathbb{R}^3$.  Therefore, it is natural to investigate decoupling inequality for degenerate hypersurfaces in $\mathbb{R}^4$.
Given a function $f\in L^{1}([0,1]^3)$ and each subset $Q$ of $[0,1]^3$, we consider the following Fourier extension operator
\[\mathcal{E}_{Q}f(x):=\int_{Q}f(\vec{\xi})e(x_1\xi_1+x_2\xi_2+x_3\xi_3+x_4(\xi^4_1+\xi^4_2+\xi^4_3))d\vec{\xi},\]
where $e(t)=e^{2\pi\iota t}$, $\vec{\xi}=(\xi_1,\xi_2,\xi_3)$ and $d\vec\xi=d\xi_1d\xi_2 d\xi_3$.
The hypersurface $\mathcal{S}^{n-1}_{2}:=\{(\xi_1,\cdots,\xi_{n-1},\sum^{n-1}_{i=1}\xi^{2}_i)\}$ denotes the paraboloid in $\mathbb{R}^n$. The Fourier extension operator associated with the paraboloid $\mathcal{S}^{n-1}_{2}$ and defined on the unit cube $[0,1]^{n-1}$ is denoted by $\mathcal{E}^{Par}_{[0,1]^{n-1}}$. Bourgain and Demeter proved the following result
\begin{theorem}[\cite{BD}]\label{BD}
    For $2\leq p\leq \frac{2(n+1)}{n-1}$ and any $\epsilon>0$ the following estimate holds
     \[\Vert \mathcal{E}^{Par}_{[0,1]^{n-1}}f\Vert_{L^{p}(B_R)}\leq C(\epsilon,p)R^{\epsilon}\Big(\sum_{\theta: R^{-1/2}-\text{cubes in }~[0,1]^{n-1}}\Vert \mathcal{E}^{Par}_{\theta}f\Vert^2_{L^{p}(\omega_{B_R})}\Big)^{\frac{1}{2}},\]
\end{theorem}
where the weight function $\omega_{B_{R}}$ associated with a ball $B_{R}(x_0)$ of radius $R$ and centered at $x_0$ is given by $\omega_{B_{R}}(x)=(1+\frac{|x-x_0|}{R})^{-100n}$ . 
In \cite{LiZheng}, Li and Zheng have studied decoupling inequality for finite type surfaces in $\mathbb{R}^3$. In fact, they have proved the following result
\begin{theorem}[\cite{LiZheng}]
    For $2\leq p\leq4$ and any $\epsilon>0$, there exists a constant $C(\epsilon,p)$ s.t. 
    \[\Vert \mathcal{E}_{[0,1]^2}f\Vert_{L^{p}(B_R)}\leq C(\epsilon,p)R^{\epsilon}\Big(\sum_{\theta\in \mathcal{F}_{3}(R,4)}\Vert \mathcal{E}_{\theta}f\Vert^2_{L^{p}(\omega_{B_R})}\Big)^{\frac{1}{2}},\]
\end{theorem}
where the Fourier extension operator $\mathcal{E}_{[0,1]^2}$ is associated with the phase function $\xi^4_1+\xi^4_2$ on the unit square $[0,1]^2$, and $\mathcal{F}_{3}(R,4)$ is defined in \cite{LiZheng}. Also see \cite{LiZheng2}, where the authors have studied restriction estimates for the hypersurface $\{(\xi_1,\xi_2,\xi^4_1+\xi^4_2):0\leq \xi_1,\xi_2\leq1\}$ in $\mathbb{R}^3$. 

Motivated by the work of Li and Zheng \cite{LiZheng}, we consider the following degenerate hypersurface in $\mathbb{R}^4$
 \[\mathcal{S}^3_{4}:=\{(\xi_1,\xi_2,\xi_3,\xi^4_1+\xi^4_2+\xi^4_3):0\leq \xi_j\leq1, \text{for}~j=1,2,3\}.\]
Our goal is to prove decoupling inequality for the extension operator $\mathcal{E}_{[0,1]^3}$ associated with the phase function $\xi^4_1+\xi^4_2+\xi^4_3$. 

\subsection{Notations and main results}
Given $1\ll R$, we partition the unit interval $[0,1]$ as \[[0,1]=\cup_{k}\theta_k, \]
where $\theta_0=[0,R^{-1/4}]$, and $\theta_k=[2^{k-1}R^{-1/4},2^{k}R^{-1/4}]$ for $1\leq k\leq \frac{1}{4}\log_{2}R$. Further, we decompose each $\theta_k$ as follows
\begin{align}\label{eq0}
\theta_k=\cup^{2^{2(k-1)}}_{\mu=1}\theta_{k,\mu},
\end{align}
where $\theta_{k,\mu}=[2^{k-1}R^{-1/4}+(\mu-1)2^{-(k-1)}R^{-1/4},2^{k-1}R^{-1/4}+\mu2^{-(k-1)}R^{-1/4}]$. Then we define 
\begin{eqnarray*}
\mathcal{F}_{4}(R,4)=\Big\{\theta_{k_1,\mu_1}\times \theta_{k_2,\mu_2}\times \theta_{k_3,\mu_3},\theta_{0}\times \theta_{k_2,\mu_2}\times \theta_{k_3,\mu_3},\theta_{k_1,\mu_1}\times \theta_{0}\times \theta_{k_3,\mu_3},\\
\theta_{k_1,\mu_1}\times \theta_{k_2,\mu_2}\times \theta_{0}, 
\theta_{k_1,\mu_1}\times \theta_{0}\times \theta_{0},\theta_{0}\times \theta_{k_2,\mu_2}\times \theta_{0},\theta_{0}\times \theta_{0}\times \theta_{k_3,\mu_3},\theta_0\times \theta_0\times \theta_0\Big\},
\end{eqnarray*}
where $1\leq k_i\leq \frac{1}{4}\log_2 R$ and $1\leq \mu_i\leq 2^{2(k_i-1)}$ for $i=1,2,3$.
The similar types of  partition can be found in \cite{LiZheng2}.
Now we state our main result
\begin{theorem}\label{Maintheorem}
    For $2\leq p\leq \frac{10}{3}$ and any $\epsilon>0$, there exists a constant $C(\epsilon,p)$ such that 
    \begin{align}\label{mainineq}
      \Vert \mathcal{E}_{[0,1]^3}f\Vert_{L^{p}(B_R)}\leq C(\epsilon,p)R^{\epsilon}\Big(\sum_{\theta\in \mathcal{F}_{4}(R,4)}\Vert \mathcal{E}_{\theta}f\Vert^2_{L^{p}(\omega_{B_R})}\Big)^{\frac{1}{2}}.  
    \end{align}
\end{theorem}
\begin{remark}
    It is important to note that the range of $p$ in Theorem  \ref{Maintheorem} is same as the range of $p$ in Theorem \ref{BD} for dimension $n=4$.
\end{remark}

Let $m\geq2$ be an integer. Then we consider the operator  \[\mathcal{E}^{m}_{[0,1]^3}f(x):=\int_{[0,1]^3}f(\vec{\xi})e(x_1\xi_1+x_2\xi_2+x_3\xi_3+x_4(\xi^{2m}_1+\xi^{2m}_2+\xi^{2m}_3))d\vec{\xi}.\]
\begin{remark}\label{powerm}
     For $2\leq p\leq \frac{10}{3}$ and any $\epsilon>0$, there exists a constant $C(\epsilon,p)$ such that 
    \begin{align*}
      \Vert \mathcal{E}^{m}_{[0,1]^3}f\Vert_{L^{p}(B_R)}\leq C(\epsilon,p)R^{\epsilon}\Big(\sum_{\theta\in \mathcal{F}_{4}(R,2m)}\Vert \mathcal{E}^{m}_{\theta}f\Vert^2_{L^{p}(\omega_{B_R})}\Big)^{\frac{1}{2}}.  
    \end{align*}
\end{remark}
The proof of the remark \ref{powerm} can be completed using the similar proof of Theorem \ref{Maintheorem}.
Observe that the Gaussian curvature of the hypersurface is \[\kappa=12^3 (\xi_1\xi_2\xi_3)^2.\]
Therefore, the curvature vanishes on the planes $\xi_1=0,\xi_2=0$ or $\xi_3=0$. The main concepts we aim to use to prove the theorem are the decoupling inequality of perturbed paraboloid and induction on scales. Therefore, 
we decompose the unit cube $[0,1]^3$ into following eight pieces. Let $1\ll K\ll R^{\epsilon}$ for any fixed $\epsilon>0$, such that $K^{-1/4}, R^{-1/4}$  and $(\frac{R}{K})^{-1/4}$ are dyadic numbers. Then 
\[[0,1]^3=\cup^7_{i=0}\Omega_i,\] where 
\begin{itemize}
\item $\Omega_0:=[0,K^{-1/4}]^3$,
\item $\Omega_1=[K^{-1/4},1]\times [0,K^{-1/4}]^2$, $\Omega_2=[0,K^{-1/4}]\times [K^{-1/4},1]\times [0,K^{-1/4}]$, $\Omega_3=[0,K^{-1/4}]^2\times [K^{-1/4},1]$,
\item $\Omega_4=[0,K^{-1/4}]\times [K^{-1/4},1]^2$, $\Omega_5=[K^{-1/4},1]\times [0,K^{-1/4}]\times [K^{-1/4},1]$, $\Omega_6=[K^{-1/4},1]^2\times [0,K^{-1/4}]$,
\item  $\Omega_7=[K^{-1/4},1]^3$.
\end{itemize}

Using Minkowski's inequality and the Cauchy-Schwarz inequality we get 
\[\Vert \mathcal{E}_{[0,1]^3}f\Vert_{L^{p}(B_R)}\leq 2^{3/2}\Big(\sum^7_{j=0}\Vert \mathcal{E}_{\Omega_j}f\Vert^2_{L^{p}(B_R)}\Big)^{1/2}.\]
Note that unlike \cite{LiZheng}, we have more Fourier extension operators with degenerate hypersurfaces.
Therefore, we need to deal with each operator separately. 
We shall prove Theorem \ref{Maintheorem} in the following three sections.

\section{Estimate of $\mathcal{E}_{\Omega_7}f$}
     We exploit the ideas of \cite{BD} and \cite{LiZheng}. We employ Bourgain and Demeter's decoupling for perturbed paraboloid. In order to do that we divide the region $\Omega_7$  into small pieces $\Omega_{7}=\cup \Omega_{\lambda_1,\lambda_2,\lambda_3}$, where $\Omega_{\lambda_1,\lambda_2,\lambda_3}=\{(\xi_1,\xi_2,\xi_3):\lambda_j\leq \xi_j\leq 2\lambda_j\}$ for $\lambda_j\in [K^{-1/4},1/2]$ and $j=1,2,3$. Note that $\Omega_7:=[K^{-1/4},1]\times [K^{-1/4},1]\times [K^{-1/4},1]$. Therefore, the Gaussian curvature of $\mathcal{S}^{3}_{4}$ is almost a constant and non-zero in the region $\Omega_{\lambda_1,\lambda_2,\lambda_3}$. Further, we divide each $\Omega_{\lambda_1,\lambda_2,\lambda_3}$  into small pieces
     
     \begin{eqnarray*}
     \Omega_{\lambda_1,\lambda_2,\lambda_3}
     =\bigcup_{j,l,m} I^{j}_{\lambda_1}\times I^{l}_{\lambda_2}\times I^{m}_{\lambda_3}
     :=\bigcup \tau^{j,l,m}_{\lambda_1,\lambda_2,\lambda_3},
     \end{eqnarray*}
     where 
     \begin{eqnarray*}
     &&I^{j}_{\lambda_1}=[\lambda_1+\frac{j-1}{\lambda_1 K^{1/2}},\lambda_1+\frac{j}{\lambda_1 K^{1/2}}],~1\leq j\leq \lambda_1K^{1/2},\\
     &&I^{l}_{\lambda_2}= [\lambda_2+\frac{l-1}{\lambda_2 K^{1/2}},\lambda_2+\frac{l}{\lambda_2 K^{1/2}}],~ 1\leq l\leq \lambda_2K^{1/2},\\
     &&I^{m}_{\lambda_3}= [\lambda_3+\frac{m-1}{\lambda_3 K^{1/2}},\lambda_3+\frac{m}{\lambda_3 K^{1/2}}],~1\leq m\leq \lambda_3K^{1/2}.
     \end{eqnarray*}
    
     We abbreviate each $\tau^{j,l,m}_{\lambda_1,\lambda_2,\lambda_3}$ by $\tau$ and w.l.o.g. we consider 
     \begin{align}\label{tauwithK}
         \tau=[\lambda_1,\lambda_1+\lambda^{-1}_1K^{-1/2}]\times [\lambda_2,\lambda_2+\lambda^{-1}_2K^{-1/2}]\times [\lambda_3,\lambda_3+\lambda^{-1}_3K^{-1/2}].
     \end{align} 
     The proof for other $\tau$ can be followed by similar methods.
     Now applying the following change of variables 
     \begin{eqnarray}\label{change7}
         \xi_j=\lambda_j+\lambda^{-1}_jK^{-1/2}\eta_j,~~~\text{for}~~j=1,2,3
     \end{eqnarray}
     we get 
     \begin{eqnarray*}
|\mathcal{E}_{\tau}f(x)|&=&\Big|\int_{[0,1]^3}\tilde{f}(\vec{\eta})e(\tilde{x_1}\eta_1+\tilde{x_2}\eta_2+\tilde{x_3}\eta_3+\tilde{x_4}\psi(\vec{\eta}))d\vec{\eta}\Big|\\
&=:&|\mathcal{E}^{pertp}_{[0,1]^3}\tilde{f}(\tilde{x})|,
     \end{eqnarray*}
     where
     \begin{eqnarray*}
\tilde{x}&=&\mathcal{T}_{7}(x),~~\tilde{x}_j=\lambda_jK^{-1/2}x_j+4\lambda^{2}_jK^{-1/2}x_4,~j=1,2,3,~\text{and}~ \tilde{x}_4=K^{-1}x_4,\\
\tilde{f}(\vec{\eta})&=&(\lambda_{1}\lambda_{2}\lambda_{3})^{-1}K^{-3/2}f(\lambda_1+\lambda^{-1}_1K^{-1/2}\eta_1,\lambda_2+\lambda^{-1}_2K^{-1/2}\eta_2,\lambda_3+\lambda^{-1}_3K^{-1/2}\eta_3),\\
 \psi(\vec{\eta})&=&\sum^{3}_{j=1}(6\eta^2_j+4\lambda^{-2}_jK^{-1/2}\eta^3_j+\lambda^{-4}_jK^{-1}\eta^{4}_j)
     \end{eqnarray*}
     and $\mathcal{E}^{pertp}_{[0,1]^3}$ denotes the Fourier extension operator associated with the phase function $\psi$. 
     We claim that for $2\leq p\leq \frac{10}{3}$ and any $\epsilon>0$, 
     \begin{align}\label{Omega7estimate}
         \Vert \mathcal{E}_{\Omega_7}f\Vert_{L^{p}(B_R)}\leq C(\epsilon,p)R^{\epsilon}\Big(\sum_{\theta\in \mathcal{F}_{4}(R,4)}\Vert \mathcal{E}_{\theta}f\Vert^2_{L^{p}(\omega_{B_R})}\Big)^{\frac{1}{2}}.
     \end{align}
     Since \[\Omega_7=\bigcup_{\lambda_1,\lambda_2,\lambda_3}\bigcup_{j,l,m}\tau^{j,l,m}_{\lambda_1,\lambda_2,\lambda_3},\]
     therefore using Minkowski's integral inequality and Cauchy--Schwarz inequality we get 
     \begin{align*}
         \Vert \mathcal{E}_{\Omega_7}f\Vert_{L^{p}(B_K)}\lesssim K^{3/4}\Big(\sum_{\tau\subset\Omega_7}\Vert \mathcal{E}_{\tau}f\Vert^2_{L^{p}(B_{K})}\Big)^{\frac{1}{2}}.
     \end{align*}
     Summing over all balls $B_{K}\subset B_{R}$ we get, 
     \begin{align*}
         \Vert \mathcal{E}_{\Omega_7}f\Vert_{L^{p}(B_R)}\lesssim K^{3/4}\Big(\sum_{\tau\subset\Omega_7}\Vert \mathcal{E}_{\tau}f\Vert^2_{L^{p}(B_{R})}\Big)^{\frac{1}{2}}.
     \end{align*}
     Now, applying the above change of variables \eqref{change7} we get 
     \[\Vert \mathcal{E}_{\tau}f\Vert^p_{L^{p}(B_{R})}=\lambda_1\lambda_2\lambda_3 K^{5/2}\Vert\mathcal{E}^{pertp}_{[0,1]^{3}}\tilde{f}\Vert^{p}_{L^{p}(\mathcal{T}_{7}(B_{R}))},\]
where $\tilde{f}$ is defined above. 
Observe that the phase function $\psi$ corresponding to the operator $\mathcal{E}^{pertp}_{[0,1]^{3}}$ can be written as $\psi(\vec{\eta})=\sum^{3}_{j=1}\phi_j(\eta_j)$. Since $\lambda_j\geq K^{-1/4}$ and $0\leq \xi_j\leq1$, therefore  each $\phi_j$ satisfies the following non-degenerate conditions 
\begin{align}\label{positivedefinite}
   \phi_j''\sim 1,~ |\phi_j^{(3)}|\lesssim 1,~ |\phi_j^{(4)}|\lesssim 1~\text{and} ~|\phi_j^{(l)}|=0,~~~\text{for} ~~l\geq5.
    \end{align}
 Hence, applying Bourgain--Demeter's decoupling from Section $7$ of \cite{BD} we get for $2\leq p\leq\frac{10}{3}$,
\begin{align*}
\Vert\mathcal{E}^{pertp}_{[0,1]^{3}}\tilde{f}\Vert_{L^{p}(B_{\frac{R}{K}})}\leq C({\epsilon},p)(\frac{R}{K})^{\epsilon} \Big(\sum_{\tilde{\theta}:K^{1/2}R^{-1/2}~\text{cubes}}\Vert \mathcal{E}^{pertp}_{\tilde{\theta}}\tilde{f}\Vert_{L^{p}(\omega_{B_{\frac{R}{K}}})}\Big)
\end{align*}
Note that $\mathcal{T}_{7}(B_{R})$ can be covered by $B_{\frac{R}{K}}$ balls.
     Now summing over all $B_{\frac{R}{K}}\subset \mathcal{T}_{7}(B_R)$ and applying the inverse change of variables we deduce the claim \eqref{Omega7estimate}. More details can be found in Lemma $3.1$ of \cite{LiZheng}.
     
    \section{Estimate of $\mathcal{E}_{\Omega_0}f$}  This can be proved by similar method as \cite{LiZheng}.
For the sake of self containment we give a sketch of the proof. We employ re-scaling in order to prove the theorem. Using the following change of variables
$\xi_j=K^{-1/4}\eta_j$ for $j=1,2,3$, we get 
\begin{eqnarray*}
    |\mathcal{E}_{\Omega_0}f(x)|&=&\Big|\int_{[0,1]^3}\tilde{f}(\vec{\eta})e(\tilde{x}_1\eta_1+\tilde{x}_2\eta_2+\tilde{x}_3\eta_3+\tilde{x}_4(\eta^4_1+\eta^4_2+\eta^4_3))d\vec{\eta}\Big|\\
    &=:&|\mathcal{E}_{[0,1]^{3}}\tilde{f}(\tilde{x})|,
\end{eqnarray*}
where 
\begin{eqnarray*}
    &&\tilde{x}=(K^{-1/4}x_1,K^{-1/4}x_2,K^{-1/4}x_3,K^{-1}x_4),\\
    &&\tilde{f}(\vec{\eta})=K^{-3/4}f(K^{-1/4}\eta_1,K^{-1/4}\eta_2,K^{-1/4}\eta_3).
\end{eqnarray*}
    Using re-scaling and the idea of Lemma $3.2$ of \cite{LiZheng} we get
\[\Vert \mathcal{E}_{\Omega_0}f\Vert_{L^{p}(B_R)}\leq D_{p}(\frac{R}{K})\Big(\sum_{\theta\subset\Omega_0}\Vert \mathcal{E}_{\theta}f\Vert^2_{L^{p}(\omega_{B_R})}\Big)^{\frac{1}{2}},\]
where ${D}_p(R)$ denote the smallest constant s.t. inequality \eqref{mainineq} holds and $\theta\in \mathcal{F}_{4}(R,4)$.
    
    \section{Estimates of $\mathcal{E}_{\Omega_1}f$, $\mathcal{E}_{\Omega_2}f$ and $\mathcal{E}_{\Omega_3}f$} 
    We need to deal with $\mathcal{E}_{\Omega_1}f$ only. The other two operators $\mathcal{E}_{\Omega_2}$ and $\mathcal{E}_{\Omega_3}$ can be handled with similar arguments due to symmetry of the regions $\Omega_{j}$, for $j=1,2,3$.
    We decompose the region $\Omega_1$ into small dyadic pieces \[\Omega_1:=\bigcup_{\lambda} \Omega_{1,\lambda},~~ \Omega_{1,\lambda} =[\lambda,2\lambda]\times [0,K^{-1/4}]\times [0,K^{-1/4}] .\]
    Further, we decompose $\Omega_{1,\lambda}$ as \[\Omega_{1,\lambda}=\bigcup_{1\leq j\leq \lambda^2K^{1/2}} [\lambda+\frac{j-1}{\lambda K^{1/2}}, \lambda+\frac{j}{\lambda K^{1/2}}]\times [0,K^{-1/4}]\times [0,K^{-1/4}]:=\bigcup_{1\leq j\leq \lambda^2K^{1/2}}  \tau^j_{\lambda}.\]
Now we claim the following lemma.
\begin{lemma}\label{Lemma3.3}
    For $2\leq p\leq6$ and any $\epsilon>0$, there exists a constant $C({\epsilon},p)>0$ such that
    \begin{align}\label{eq3.10}
    \Vert \mathcal{E}_{\Omega_1}f\Vert_{L^{p}(B_R)}\leq C({\epsilon},p)K^{\epsilon} \Big(\sum_{\lambda}\sum_{\tau_{\lambda}\subset \Omega_{1,\lambda}}\Vert \mathcal{E}_{\tau_{\lambda}}f\Vert^{2}_{L^{p}(\omega_{B_{R}})}\Big)^{1/2}.
    \end{align}
\end{lemma}
\begin{proof}[Proof of Lemma\ref{Lemma3.3}]
    In order to prove the above lemma, we freeze two variables $x_2$ and $x_3$. Indeed, we consider a bump function $\varphi\in C^{\infty}_{c}(\mathbb{R}^4)$ with $supp(\varphi)\subset B(0,1)$ and $|\check{\varphi}(x)|\geq 1$ for all $x\in B(0,1)$. Define $F:=\check{\varphi}_{K^{-1}}\mathcal{E}_{\Omega_1}f$, where $\varphi_{K^{-1}}(y)=K^4\varphi(Ky)$ for $y\in \mathbb{R}^4$. We define $G:=F(\cdot,x_2,x_3,\cdot)$. From Lemma $4.1$ of \cite{Guthlec7} we get that $supp(\hat{G})$ is contained in the projection of $supp(\hat{F})$ on the plane $\xi_2=\xi_3=0$, i.e. in the $K^{-1}$ neighborhood of the curve $\Gamma_{\lambda}=\{(t,t^4):t\in [\lambda,2\lambda]\}$. Therefore, invoking Lemma $3.4$ of \cite{LiZheng} (see \cite{Yang} for the same lemma)  we get,
    \[\Vert F(\cdot,x_2,x_3,\cdot)\Vert_{L^{p}(\mathbb{R}^2)}\leq C({\epsilon},p)K^{\epsilon}\Big(\sum_{\tau}\Vert G_{\tau}(\cdot,x_2,x_3,\cdot)\Vert^2_{L^{p}(\mathbb{R}^2)}\Big)^{1/2},\]
    for $2\leq p\leq6$.
    Therefore, integrating both sides w.r.t $x_2,x_3$ variables we get 
    \[\Vert F\Vert_{L^{p}(\mathbb{R}^4)}\leq C({\epsilon},p)K^{\epsilon}\Big(\sum_{\tau}\Vert F_{\tau}\Vert^2_{L^{p}(\mathbb{R}^4)}\Big)^{1/2}.\]
    Thus, we have 
    \begin{eqnarray*}
        \Vert \mathcal{E}_{\Omega_{1,\lambda}}f\Vert_{L^{p}(B_K)}\lesssim \Vert F\Vert_{L^{p}(\mathbb{R}^4)}&\leq& C({\epsilon},p)K^{\epsilon}\Big(\sum_{\tau}\Vert F_{\tau}\Vert^2_{L^{p}(\mathbb{R}^4)}\Big)^{1/2}\\
        &\leq& C({\epsilon},p)K^{\epsilon}\Big(\sum_{\tau_{\lambda}\subset \Omega_{1,\lambda}}\Vert \mathcal{E}_{\tau_{\lambda}}f\Vert^2_{L^{p}(\omega_{B_{K}})}\Big)^{1/2}.
    \end{eqnarray*}
    Further, taking sum over $\lambda$ we get 
    \[\Vert \mathcal{E}_{\Omega_{1}}f\Vert_{L^{p}(B_K)}\leq C({\epsilon},p)K^{\epsilon}\Big(\sum_{\lambda}\sum_{\tau_{\lambda}\subset \Omega_{1,\lambda}}\Vert \mathcal{E}_{\tau_{\lambda}}f\Vert^2_{L^{p}(\omega_{B_{K}})}\Big)^{1/2}.\]
    Now, considering the sum over all balls  $B_K\subset B_{R}$ we get Lemma \ref{Lemma3.3}.
\end{proof}
Next, we need to estimate $\Vert \mathcal{E}_{\tau_{\lambda}}f\Vert_{L^{p}({B_{R}})}$ for the typical case of $\tau=[\lambda,\lambda+\lambda^{-1}K^{-1/2}]\times [0,K^{-1/4}]^2$. 
The estimate for other $\tau$ can be deduced by the same arguments. We apply the following change of variables
\begin{eqnarray}\label{change1}
    \xi_1=\lambda+\frac{\eta_1}{\lambda K^{1/2}},~ \xi_2=\frac{\eta_2}{K^{1/4}}, ~ \xi_3=\frac{\eta_3}{K^{1/4}}.
\end{eqnarray}

Therefore, we get \[|\mathcal{E}_{\tau}f(x)|=\Big|\int_{[0,1]^3}\tilde{f}(\vec{\eta})e(\tilde{x}_1\eta_1+\tilde{x}_2\eta_2+\tilde{x}_3\eta_3+\tilde{x}_4\psi_1(\vec{\eta}))d\vec{\eta}\Big|:=\Big|\tilde{\mathcal{E}}_{[0,1]^3}\tilde{f}(\tilde{x})\Big|,\]
where 
\begin{eqnarray*}
    &&\tilde{x}=(\lambda^{-1}K^{-1/2}x_1+4\lambda^2K^{-1/2}x_4, K^{-1/4}x_2,K^{-1/4}x_3, K^{-1}x_4),\\
    &&\tilde{f}(\vec{\eta})=\lambda^{-1}K^{-1}f(\lambda+\frac{\eta_1}{\lambda K^{1/2}},\frac{\eta_2}{K^{1/4}},\frac{\eta_3}{K^{1/4}}),\\
    &&\psi_{1}(\vec{\eta})=(6\eta^{2}_{1}+4\lambda^{-2}K^{-1/2}\eta^{3}_{1}+\lambda^{-4}K^{-1}\eta^{4}_{1})+\eta^{4}_{2}+\eta^{4}_{3}.
\end{eqnarray*}
Therefore, the extension operator $\tilde{\mathcal{E}}_{[0,1]^3}$ is associated with the 

hypersurface $\mathcal{S}_{1,4}:=\{(\xi_1,\xi_2,\xi_3,\phi_{1}(\xi)+\xi^{4}_2+\xi^{4}_3): (\xi_1,\xi_2,\xi_3)\in [0,1]^{3}\}$, with $\phi_1$ on  $ [0,1]$ satisfying the following properties 
\begin{eqnarray*}
   \phi_1''\sim 1,~ |\phi_1^{(3)}|\lesssim 1,~ |\phi_1^{(4)}|\lesssim 1~\text{and} ~|\phi_1^{(j)}|=0,~~~\text{for} ~~j\geq5.   
\end{eqnarray*}
 Using the decomposition \eqref{eq0} we define a new collection of sets 
\begin{align*}
    \mathcal{F}_4{(R,1,4)}:=\Big\{[a,a+R^{-1/2}]\times \theta_{k_2,\mu_2}\times \theta_{k_3,\mu_3}, [a,a+R^{-1/2}]\times \theta_{k_2,\mu_2}\times \theta_{0},\\
    [a,a+R^{-1/2}]\times \theta_{0}\times \theta_{k_2,\mu_2},[a,a+R^{-1/2}]\times \theta_{0}\times \theta_{0}\Big\},
\end{align*}
    where $a\in [0,1-R^{-1/2}]\cap R^{-1/2}\mathbb{Z}$,  $1\leq k_i\leq \frac{1}{4}\log_{2}R$ and $1\leq \mu_i\leq 2^{2(k_i-1)}$.
We claim the following lemma.
\begin{lemma}\label{Lemma3.5}
For $2\leq p\leq \frac{10}{3}$ and any $\epsilon>0$ , there exists a constant $C(\epsilon,p)$ s.t. 
    \begin{align}\label{lemma3.5}
     \Vert \tilde{\mathcal{E}}_{[0,1]^3}f\Vert_{L^{p}(B_{R})}\leq C({\epsilon},p)R^{\epsilon}\Big(\sum_{\vartheta \in\mathcal{F}_{4}(R,1,4)} \Vert \tilde{\mathcal{E}}_{\vartheta}f\Vert^{2}_{L^{p}(\omega_{B_{R}})}\Big)^{1/2}.
\end{align}
\end{lemma}

Now assuming Lemma \eqref{Lemma3.5} for a moment we want to prove that for each $\tau_{\lambda}\subset \Omega_{1,\lambda}$, 
\begin{align}\label{hatch}
     \Vert {\mathcal{E}}_{\tau_{\lambda}}f\Vert_{L^{p}(B_{R})}\leq C(\epsilon,p)(\frac{R}{K})^{\epsilon}\Big(\sum_{\theta \subset \tau_{\lambda}} \Vert {\mathcal{E}}_{\theta}f\Vert^{2}_{L^{p}(\omega_{B_{R}})}\Big)^{1/2}.
\end{align}
In order to prove the above estimate we consider $\tau_{\lambda}=[\lambda,\lambda+\lambda^{-1}K^{-1/2}]\times [0,K^{-1/4}]^2$. Considering the change of variables \eqref{change1} we get
\begin{align*}
    \Vert \tilde{\mathcal{E}}_{\tau_{\lambda}}f\Vert_{L^{p}(B_{R})}=K^2\lambda \Vert\tilde{\mathcal{E}}_{[0,1]^3}\tilde{f}\Vert_{L^{p}(\mathcal{T}(B_R))},
\end{align*}
where $\tilde{f}(\vec{\eta})=\lambda^{-1}K^{-1}f(\lambda+\lambda^{-1}K^{-1/2}\eta_1,K^{-1/4}\eta_2,K^{-1/4}\eta_3)$ and $\mathcal{T}({x})=(\lambda^{-1}K^{-1/2}x_1+4\lambda^2K^{-1/2}x_4,K^{-1/4}x_2,K^{-1/4}x_3,K^{-1}x_4)$. Observe that the size of the ball $B_R$ under the translation and dilation map $\mathcal{T}$ is roughly $\lambda^{-1/2}K^{-1/2}R\times K^{-1/4}R\times K^{-1/4}R\times K^{-1}R$, which can be covered by $B_{\frac{R}{K}}$ balls. Now we can apply Lemma \ref{Lemma3.5} on each $B_{\frac{R}{K}}$ balls to estimate $\Vert \tilde{\mathcal{E}}_{[0,1]^3}\tilde{f}\Vert_{L^{p}(B_{\frac{R}{K}})}$, and  
in order to do that we need to check that the image of $\theta\subset \tau_{\lambda}$ is $\tilde{\theta}\in \mathcal{F}_{4}(R/K, 1,4)$ under the change of variables \eqref{change1}. The above things can be checked using  the similar arguments of case (a) and case (b) of \cite{LiZheng}. Therefore, we obtain 
\begin{align*}
\Vert \tilde{\mathcal{E}}_{[0,1]^3}\tilde{f}\Vert_{L^{p}(B_{\frac{R}{K}})}\leq C({\epsilon},p)(\frac{R}{K})^{\epsilon}\Big(\sum_{\tilde{\theta}\in \mathcal{F}_{4}(R/K,1,4)}\Vert \tilde{\mathcal{E}}_{\tilde{\theta}}\tilde{f}\Vert^{2}_{L^{p}(\omega_{B_{\frac{R}{K}}})}\Big)^{1/2}.
\end{align*}
Summing over all balls $B_{\frac{R}{K}}\subset \mathcal{T}(B_{R})$ and using Minkowski's integral inequality we get 

\begin{align*}
\Vert \tilde{\mathcal{E}}_{[0,1]^3}\tilde{f}\Vert_{L^{p}(\mathcal{T}(B_{R}))}\leq C({\epsilon},p)(\frac{R}{K})^{\epsilon}\Big(\sum_{\tilde{\theta}\in \mathcal{F}_{4}(R/K,1,4)}\Vert \tilde{\mathcal{E}}_{\tilde{\theta}}\tilde{f}\Vert^{2}_{L^{p}(\mathcal{T}(B_{R}))}\Big)^{1/2}.
\end{align*}
Taking the inverse change of variable we have 

\begin{eqnarray*}
\Vert {\mathcal{E}}_{\tau_{\lambda}}{f}\Vert_{L^{p}(B_{R})}\leq C({\epsilon},p)(\frac{R}{K})^{\epsilon}\Big(\sum_{\theta\subset\tau_{\lambda}}\Vert {\mathcal{E}}_{{\theta}}{f}\Vert^{2}_{L^{p}(\omega_{B_{R}})}\Big)^{1/2}.
\end{eqnarray*}

This proves \eqref{hatch}. Now plugging \eqref{hatch} in \eqref{eq3.10} we get 
\begin{eqnarray*}
    \Vert \mathcal{E}_{\Omega_1}f\Vert_{L^{p}(B_R)}&\leq &C({\epsilon},p)K^{\epsilon}(\frac{R}{K})^{\epsilon}\Big(\sum_{\lambda}\sum_{\tau_{\lambda}}\sum_{\theta\subset\tau_{\lambda}}\Vert \mathcal{E}_{\theta}f\Vert^{2}_{L^{p}(\omega_{B_{R}})}\Big)^{1/2}\\
&\leq&C({\epsilon},p)R^{\epsilon}\Big(\sum_{\theta\subset\Omega_1}\Vert \mathcal{E}_{\theta}f\Vert^{2}_{L^{p}(\omega_{B_{R}})}\Big)^{1/2},
\end{eqnarray*}
for $\theta\in \mathcal{F}_{4}(R,4)$.

~~~~~~~~~~~~~~~~~
Now it remains to prove Lemma \ref{Lemma3.5}. 
\begin{proof}[Proof of Lemma \ref{Lemma3.5}]
Denote the least constant $\tilde{D}_{p}(R)$ such that  \eqref{lemma3.5} holds. 
 We decompose the cube $[0,1]^3$ into three regions $[0,1]^3=\mathcal{D}_{1}\cup\mathcal{D}_{2}\cup\mathcal{D}_{3}$, where 
\begin{align*}
    &&\mathcal{D}_{1}=[0,1]\times [K^{-1/4},1]^2, ~ \mathcal{D}_{2}=[0,1]\times [0,K^{-1/4}]\times [0,1],\\
    &&\mathcal{D}_{3}=[0,1]\times[K^{-1/4},1]\times [0,K^{-1/4}].
\end{align*}
Therefore, \[\Vert \tilde{\mathcal{E}}_{[0,1]^3}{f}\Vert_p\leq \Vert \tilde{\mathcal{E}}_{\mathcal{D}_1}{f}\Vert_p+\Vert \tilde{\mathcal{E}}_{\mathcal{D}_2}{f}\Vert_p+\Vert \tilde{\mathcal{E}}_{\mathcal{D}_3}{f}\Vert_p.\]

\subsection{Estimation of $\Vert\tilde{\mathcal{E}}_{\mathcal{D}_1}f\Vert_{L^{p}(B_R)}$}\label{ED1}
Trivially for any $p\geq2$ we have 
\[\Vert\tilde{\mathcal{E}}_{\mathcal{D}_1}f\Vert_{L^{p}(B_K)}\lesssim K^{O(1)}\Big(\sum_{\nu\subset\mathcal{D}_1}\Vert \tilde{\mathcal{E}}_{\nu}f\Vert^2_{L^{p}(B_K)}\Big)^{1/2},\]
for all $\nu\in\mathcal{F}_{4}(K,1,4)$. Then summing over $B_K\subset B_R$ and using Minkowski's integral inequality we get 
\begin{align}\label{eq3.20}
    \Vert\tilde{\mathcal{E}}_{\mathcal{D}_1}f\Vert_{L^{p}(B_R)}\lesssim K^{O(1)}\Big(\sum_{\nu\subset\mathcal{D}_1}\Vert \tilde{\mathcal{E}}_{\nu}f\Vert^2_{L^{p}(B_R)}\Big)^{1/2},
\end{align}
for any $\nu\subset \mathcal{D}_1$ of size $K^{-1/2}\times \sigma^{-1}K^{-1/2}\times\lambda^{-1}K^{-1/2}$, for dyadic numbers $K^{-1/4}\leq \lambda,\sigma\leq 1/2$. We claim that  for $2\leq p\leq \frac{10}{3}$,
\begin{align}\label{eq3.21}
     \Vert\tilde{\mathcal{E}}_{\nu}f\Vert_{L^{p}(B_R)}\leq C({\epsilon},p)(\frac{R}{K})^{\epsilon}\Big(\sum_{\theta\subset\nu}\Vert \tilde{\mathcal{E}}_{\theta}f\Vert^2_{L^{p}(\omega_{B_R})}\Big)^{1/2},
\end{align}
where $\theta\in \mathcal{F}_{4}(R,1,4)$. Putting \eqref{eq3.21} in \eqref{eq3.20} we get 
\begin{align*}
     \Vert\tilde{\mathcal{E}}_{\mathcal{D}_1}f\Vert_{L^{p}(B_R)}\leq C({\epsilon},p)K^{O(1)}{R}^{\epsilon}\Big(\sum_{\theta\subset\mathcal{D}_1}\Vert \tilde{\mathcal{E}}_{\theta}f\Vert^2_{L^{p}(\omega_{B_R})}\Big)^{1/2}.
\end{align*}
Now, we aim to prove \eqref{eq3.21}. W.l.o.g we assume $\nu=[0,K^{-1/2}]\times[\sigma,\sigma+\sigma^{-1}K^{-1/2}]\times [\lambda,\lambda+\lambda^{-1}K^{-1/2}] $. Applying the change of variables $\xi_1=K^{-1/2}\eta_1, \xi_2=\sigma+\sigma^{-1}K^{-1/2}\eta_2,\xi_3=\lambda+\lambda^{-1}K^{-1/2}\eta_3$,  we get 
\begin{align*}
    \Vert \tilde{\mathcal{E}}_{\nu}f\Vert^{p}_{L^{p}(B_{R})}=\sigma\lambda K^{5/2}\Vert \mathcal{E}^{pertp}_{[0,1]^3}\tilde{f}\Vert^{p}_{L^{p}(\mathcal{T}_{1}(B_{R}))},
\end{align*}
where \[\tilde{f}(\vec{\eta})=K^{-3/2}\sigma^{-1}\lambda^{-1}f(K^{-1/2}\eta_1, \sigma+\sigma^{-1}K^{-1/2}\eta_2,\lambda+\lambda^{-1}K^{-1/2}\eta_3),\] and 
\[\mathcal{T}_{1}(x):=(K^{-1/2}x_1,K^{-1/2}\sigma^{-1}x_2+4\sigma^2K^{-1/2}x_4,K^{-1/2}\lambda^{-1}x_3+4\lambda^2K^{-1/2}x_4, K^{-1}x_4).\]
Therefore, the size of $\mathcal{T}_{1}(B_{R})$ is roughly $K^{-1/2}R\times \sigma^{-1}K^{-1/2}R\times \lambda^{-1}K^{-1/2}R\times K^{-1}R$, which can be covered by finitely overlapping balls $B_{\frac{R}{K}}$.
Now applying decoupling theory for perturbed paraboloid \cite{BD} in dimension $n=4$ w.r.t. each ball $B_{\frac{R}{K}}$ we get that for $2\leq p\leq \frac{10}{3}$, 
\begin{align*}
    \Vert \mathcal{E}^{pertp}_{[0,1]^3}\tilde{f}\Vert_{L^{p}(B_{\frac{R}{K}})}\leq C({\epsilon},p)(\frac{R}{K})^{\epsilon}\Big(\sum_{\tilde{\theta}:K^{1/2}R^{-1/2}-cube}\Vert \mathcal{E}^{pertp}_{\tilde{\theta}}\tilde{f}\Vert^{2}_{L^{p}(\omega_{B_{\frac{R}{K}}})}\Big)^{1/2}.
\end{align*}
Now summing over all the balls $B_{\frac{R}{K}}\subset \mathcal{T}_{1}(B_{R})$ we get 
\begin{align*}
    \Vert \mathcal{E}^{pertp}_{[0,1]^3}\tilde{f}\Vert_{L^{p}(\mathcal{T}_{1}(B_{R}))}\leq C({\epsilon},p)(\frac{R}{K})^{\epsilon}\Big(\sum_{\tilde{\theta}:K^{1/2}R^{-1/2}-cube}\Vert \mathcal{E}^{pertp}_{\tilde{\theta}}\tilde{f}\Vert^{2}_{L^{p}(\mathcal{T}_{1}(B_{R}))}\Big)^{1/2}.
\end{align*}
Now taking the inverse change of variables we deduce
\begin{align*}
    \Vert \tilde{\mathcal{E}}_{\nu}f\Vert_{L^{p}(B_{R})}\leq C({\epsilon},p)(\frac{R}{K})^{\epsilon}\Big(\sum_{\vartheta\subset\nu}\Vert \tilde{\mathcal{E}}_{{\vartheta}}\tilde{f}\Vert^{2}_{L^{p}(\omega_{B_{R}})}\Big)^{1/2},
\end{align*}
where $\vartheta\in \mathcal{F}_{4}(R,1,4)$.
This proves the claim of \eqref{eq3.21}.

\subsection{Estimate of $\Vert \tilde{\mathcal{E}}_{\mathcal{D}_{2}}f\Vert_{L^{p}(B_R)}$}\label{ED2}
Firstly, we claim that for $2\leq p\leq 6$ and any $\epsilon>0$ 
\begin{align*}
    \Vert \tilde{\mathcal{E}}_{\mathcal{D}_2}f\Vert_{L^{p}(B_R)}\leq C({\epsilon},p)K^{\epsilon}\Big(\sum_{\nu\subset\mathcal{D}_2} \Vert \tilde{\mathcal{E}}_{\nu}f\Vert_{L^{p}(\omega_{B_{R}})}\Big),
\end{align*}
where $\nu$ is a rectangle of size $K^{-1/2}\times K^{-1/4}\times K^{-1/2}$.
This can be proved using the similar proof of Lemma \ref{Lemma3.3}. Therefore, we focus on the estimate of $\Vert \tilde{\mathcal{E}}_{\nu}f\Vert_{L^{p}(B_{R})}$ and we  claim that 
\begin{eqnarray*}
    \Vert \tilde{\mathcal{E}}_{\nu}{f}\Vert_{L^{p}(B_R)}\leq \tilde{D}_{p}(\frac{R}{K})\Big(\sum_{\theta \in \mathcal{F}_{4}(R/K,1,4)} \Vert \tilde{\mathcal{E}}_{\theta}f\Vert^{2}_{L^{p}(\omega_{B_{R}})}\Big)^{1/2}, 
\end{eqnarray*}
where $\tilde{D}_{p}$ is the smallest constant s.t. inequality \eqref{lemma3.5}  holds. 
Indeed, 
w.o.l.g. we assume $\nu=[0,K^{-1/2}]\times [0,K^{-1/4}]\times [0,K^{-1/2}]$. Applying a suitable change of variables we get 
\begin{align*}
   \Vert \tilde{\mathcal{E}}_{\nu}f\Vert^{p}_{L^{p}(B_{R})}=  K^{9/4}\Vert \tilde{\mathcal{E}}_{[0,1]^3}\tilde{f}\Vert^{p}_{L^{p}(\mathcal{T}_{2}(B_{R}))},
\end{align*}
where $\tilde{f}(\vec{\eta})=K^{-5/4}f(K^{-1/2}\eta_1,K^{-1/4}\eta_2,K^{-1/2}\eta_3)$ and\\
$\mathcal{T}_{2}({x}):=(K^{-1/2}x_1,K^{-1/4}x_2,K^{-1/2}x_3,K^{-1}x_4)$. The size of $\mathcal{T}_{2}(B_{R})$ is roughly $K^{-1/2}R\times K^{-1/4}R\times K^{-1/2}R\times K^{-1}R$ and it could be covered by union of finitely overlapping balls of size $B_{\frac{R}{K}}$. By the definition of $\tilde{D}_{p}$ we get 
\begin{eqnarray*}\label{3star}
    \Vert \tilde{\mathcal{E}}_{[0,1]^3}\tilde{f}\Vert_{L^{p}(B_{\frac{R}{K}})}\leq \tilde{D}_{p}(\frac{R}{K})\Big(\sum_{\vartheta \in \mathcal{F}_{4}(R/K,1,4)} \Vert \tilde{\mathcal{E}}_{\vartheta}\tilde{f}\Vert^{2}_{L^{p}(\omega_{B_{R/K}})}\Big)^{1/2},~~\text{for}~~2\leq p\leq\frac{10}{3}.
\end{eqnarray*}
Now, first  we take the sum over all the balls $B_{\frac{R}{K}}$ contained in $\mathcal{T}_{2}(B_{R})$ and apply Minkowski's integral inequality and then apply the inverse change of variables to get 
\begin{eqnarray*}\label{4star}
    \Vert \tilde{\mathcal{E}}_{\nu}{f}\Vert_{L^{p}(B_R)}\leq \tilde{D}_{p}(\frac{R}{K})\Big(\sum_{\theta \in \mathcal{F}_{4}(R,1,4)} \Vert \tilde{\mathcal{E}}_{\theta}f\Vert^{2}_{L^{p}(\omega_{B_{R}})}\Big)^{1/2},
\end{eqnarray*}
for $2\leq p\leq\frac{10}{3}$ and $\theta\subset\nu$.
Similarly, repeating the arguments of subsection \ref{ED2} we get 
\begin{align*}
    \Vert \tilde{\mathcal{E}}_{\mathcal{D}_{3}}{f}\Vert_{L^{p}(B_R)}\leq C({\epsilon},p)K^{\epsilon} \tilde{D}_{p}(\frac{R}{K})\Big(\sum_{\theta \in \mathcal{F}_{4}(R,1,4)} \Vert \tilde{\mathcal{E}}_{\theta}f\Vert^{2}_{L^{p}(\omega_{B_{R}})}\Big)^{1/2},
\end{align*}
for $2\leq p\leq \frac{10}{3}$ and any $\epsilon>0$.
Hence, combining all the three estimates we get 
\begin{eqnarray*}\label{5star}
    \Vert \tilde{\mathcal{E}}_{[0,1]^3}{f}\Vert_{L^{p}(B_R)}\leq \Big(C({\epsilon},p)K^{O(1)}R^{\epsilon}+2C({\epsilon},p)K^{\epsilon}\tilde{D}_{p}(\frac{R}{K})\Big)\Big(\sum_{\theta \in \mathcal{F}_{4}(R,1,4)} \Vert \tilde{\mathcal{E}}_{\theta}f\Vert^{2}_{L^{p}(\omega_{B_{R}})}\Big)^{1/2}.
\end{eqnarray*}
Then the definition of $\tilde{D}_{p}$ implies 
\[\tilde{D}_{p}(R)\leq C({\epsilon},p)K^{O(1)}R^{\epsilon}+2C({\epsilon},p)K^{\epsilon}\tilde{D}_{p}(\frac{R}{K}).\]
Iterating the above for $[\log_{K}R]$ times we get $\tilde{D}_{p}(R)\leq C_{} R^{\epsilon}$.

This completes the proof of Lemma \ref{Lemma3.5}.
\end{proof}

\section{Estimates of $\mathcal{E}_{\Omega_4}f,\mathcal{E}_{\Omega_5}f$ and $\mathcal{E}_{\Omega_6}f$}
Due to symmetry, it is enough to prove estimates of $\mathcal{E}_{\Omega_6}f$. Note that the operator $\mathcal{E}_{\Omega_6}$ can be dealt with the similar approach as the estimates of $\mathcal{E}_{\Omega_1}f$. We give a sketch of the proof for the sake of completeness.
We decompose $\Omega_6$ as \[\Omega_6=\bigcup_{K^{-1/4}\leq \lambda,\sigma\leq \frac{1}{2}}[\lambda,2\lambda]\times [\sigma,2\sigma]\times [0,K^{-1/4}]:=\bigcup_{K^{-1/4}\leq \lambda,\sigma\leq\frac{1}{2}}\Omega_{\lambda,\sigma}.\]
  Further, we decompose 
  \[\Omega_{\lambda,\sigma}=\bigcup_{j,l}[\lambda+\frac{j-1}{\lambda K^{1/2}},\lambda+\frac{j}{\lambda K^{1/2}}]\times [\sigma+\frac{l-1}{\sigma K^{1/2}},\sigma+\frac{l}{\sigma K^{1/2}}]\times [0,K^{-1/4}]:=\bigcup_{j,l}\tau^{j,l}_{\lambda,\sigma} ,\]
for $1\leq j\leq \lambda^2K^{1/2}$ and $1\leq l\leq \sigma^2 K^{1/2}$.
Then using the same proof of Lemma \ref{Lemma3.3} we get that for $2\leq p\leq6$ and any $\epsilon>0$
\begin{align}\label{lemma05.1}
    \Vert\mathcal{E}_{\Omega_6}f\Vert_{L^{p}(B_R)}\leq C({\epsilon},p) K^{\epsilon} \Big(\sum_{\lambda,\sigma}\sum_{\tau_{\lambda,\sigma}}\Vert \mathcal{E}_{\tau_{\lambda,\sigma}}f\Vert^{2}_{L^{p}(\omega_{B_{R}})}\Big)^{1/2}.
\end{align}
Now we need to estimate $\Vert \mathcal{E}_{\tau_{\lambda,\sigma}}f\Vert^{2}_{L^{p}({B_{R}})}$ for each $\tau_{\lambda,\sigma}$. W.l.o.g. we assume $\tau_{\lambda,\sigma}=[\lambda,\lambda+\frac{1}{\lambda K^{1/2}}]\times [\sigma,\sigma+\frac{1}{\sigma K^{1/2}}]\times [0,K^{-1/4}] $.
Applying suitable change of variables we get 
\begin{align}\label{change2}
    |\mathcal{E}_{\tau_{\lambda,\sigma}}f(x)|=\lambda\sigma K^{9/4}|\mathcal{E'}_{[0,1]^3}\tilde{f}(\tilde{x})|,
\end{align}
where $\tilde{f}(\eta)=\lambda^{-1}\sigma^{-1}K^{-5/4}f(\lambda+\frac{1}{\lambda K^{1/2}}\eta_1,\sigma+\frac{1}{\sigma K^{1/2}}\eta_2,K^{-1/4}\eta_3)$ and $\mathcal{E'}_{[0,1]^3}$ denotes the Fourier extension operator associated with the phase function
\begin{eqnarray*}
    \psi_{6}(\vec{\eta})&=&(6\eta^2_1+4\lambda^{-2}K^{-1/2}\eta^3_1+\lambda^{-4}K^{-1}\eta^{4}_1)+(6\eta^2_2+4\sigma^{-2}K^{-1/2}\eta^3_2+\sigma^{-4}K^{-1}\eta^{4}_2)+\eta^{4}_{3}\\
    &:=&\varphi_1(\eta_1)+\varphi_{2}(\eta_2)+\eta^{4}_{3},
\end{eqnarray*}
with $\varphi_1,\varphi_2$ satisfying conditions \ref{positivedefinite}. 
Consider the collection of sets 
\begin{eqnarray*}
    \tilde{\mathcal{F}}_{4}(R,4):=\{[a,a+R^{-1/2}]\times [b,b+R^{-1/2}]\times [0,R^{-1/4}], [a,a+R^{-1/2}]\times [b,b+R^{-1/2}]\times \theta_{k,\mu} \},
\end{eqnarray*}
for $a,b\in [0,1-R^{-1/2}]\cap R^{-1/2}\mathbb{Z}, 1\leq k\leq \frac{1}{4}\log R$, and  $1\leq \mu\leq 2^{2(k-1)}$.
The following lemma holds.
\begin{lemma}\label{Lemma05.2}
For $2\leq p\leq \frac{10}{3}$ and any $\epsilon>0$, 
    \begin{align*}
        \Vert\mathcal{E'}_{[0,1]^3}f\Vert_{L^{p}(B_R)}\leq C({\epsilon},p)R^{\epsilon} \Big(\sum_{\vartheta\in\tilde{\mathcal{F}}_{4}(R,4)}\Vert \mathcal{E'}_{\vartheta}f\Vert^2_{L^{p}(\omega_{B_{R}})}\Big)^{1/2}.
    \end{align*}
\end{lemma}
The Lemma \ref{Lemma05.2} can be proved in a similar manner as Lemma \ref{Lemma3.5}. However, we shall give a sketch of the proof later. Therefore, assuming the above lemma we aim to prove that for $2\leq p\leq \frac{10}{3}$ and any $\epsilon>0$,
\begin{align}\label{lastclaim}
    \Vert \mathcal{E}_{\tau_{\lambda,\sigma}}f\Vert_{L^{p}(B_{R})}\lesssim (R/K)^{\epsilon}\Big(\sum_{\theta\subset\tau_{\lambda,\sigma}}\Vert\mathcal{E}_{\theta}f\Vert^{2}_{L^{p}(\omega_{B_{R}})}\Big)^{1/2}.
\end{align}
 Observe that using \eqref{change2} we get 
 \begin{align*}
    \Vert\mathcal{E}_{\tau_{\lambda,\sigma}}f\Vert_{L^{p}(B_{R})}=\lambda\sigma K^{9/4}\Vert\mathcal{E'}_{[0,1]^3}\tilde{f}\Vert_{L^{p}(\mathcal{T}_{6}(B_{R}))},
\end{align*}
where $\mathcal{T}_{6}(x)=\Big(\lambda^{-1}K^{-1/2}x_1+4\lambda^2K^{-1/2}x_4, \sigma^{-1}K^{-1/2}x_2+4\sigma^2K^{-1/2}x_4, K^{-1/4}x_3, K^{-1}x_{4}\Big)$.
Note that $\mathcal{T}_{6}(B_{R})$ can be covered by $B_{\frac{R}{K}}$ balls. Therefore, applying Lemma \ref{Lemma05.2} we get 
\begin{align*}
        \Vert\mathcal{E'}_{[0,1]^3}\tilde{f}\Vert_{L^{p}(B_{\frac{R}{K}})}\leq C({\epsilon},p)(\frac{R}{K})^{\epsilon} \Big(\sum_{\vartheta\in\tilde{\mathcal{F}}_{4}(R,4)}\Vert \mathcal{E'}_{\vartheta}\tilde{f}\Vert^2_{L^{p}(\omega_{B_{\frac{R}{K}}})}\Big)^{1/2}.
    \end{align*}
    It is to be noted that before applying Lemma \ref{Lemma05.2} we need to check that for any $\theta\subset\tau_{\lambda,\sigma}$, the image $\mathcal{T}_{6}(\theta)={\vartheta}\in \tilde{\mathcal{F}}_{4}(R,4)$. This could be verified using similar arguments as case $(a)$ and case $(b)$ of \cite{LiZheng}. Now summing over all $B_{\frac{R}{K}}\subset \mathcal{T}_{6}(B_{R})$  and using Minkowski's integral inequality and finally applying inverse change of variables we get the claim \eqref{lastclaim}. Now putting \eqref{lastclaim} in \eqref{lemma05.1} we get 
    \begin{align*}
        \Vert\mathcal{E}_{\Omega_6}f\Vert_{L^{p}(B_R)}\leq C_{\epsilon}K^{\epsilon}(\frac{R}{K})^{\epsilon} \Big(\sum_{\lambda,\sigma}\sum_{\tau_{\lambda,\sigma}}\sum_{\theta\subset \tau_{\lambda,\sigma}}\Vert \mathcal{E}_{\theta}f\Vert^{2}_{L^{p}(\omega_{B_{R}})}\Big)^{1/2}.
    \end{align*}
    \begin{proof}[Proof of Lemma \ref{Lemma05.2}]
        Note that $\Omega_6=[K^{-1/4},1]^2\times [0,K^{-1/4}]$. Therefore, the difficulty in estimating $\mathcal{E}_{\Omega_6}f$ arises when $\xi_3=0$. In order to overcome this difficulty we decompose the unit cube $[0,1]^3$ into two rectangles 
        \[\mathcal{D}^6_1:=[0,1]^2\times [K^{-1/4},1], \mathcal{D}^6_{2}=[0,1]^2\times [0,K^{-1/4}].\]
   
    Therefore, $\Vert\mathcal{E'}_{[0,1]^3}f\Vert_{L^{p}(B_{R})}\leq \Vert\mathcal{E'}_{\mathcal{D}^{6}_{1}}f\Vert_{L^{p}(B_{R})}+\Vert\mathcal{E'}_{\mathcal{D}^{6}_{2}}f\Vert_{L^{p}(B_{R})}$.
    It is not hard to see that the estimate of $\Vert\mathcal{E'}_{\mathcal{D}^{6}_{1}}f\Vert_{L^{p}(B_{R})}$ can be deduced using the similar arguments of subsection \ref{ED1} with appropriate change of variables. On the other hand,  the estimate of $\Vert\mathcal{E'}_{\mathcal{D}^{6}_{2}}f\Vert_{L^{p}(B_{R})}$ can be proved using similar arguments of subsection \ref{ED2} with appropriate change of variables. This completes the proof of Theorem \ref{Maintheorem}.
     \end{proof}
     
     \section*{Acknowledgements}
     The author is profoundly thankful to Prof. Saurabh Shrivastava for encouraging him to study decoupling theory. The author is grateful to Prof. Sanghyuk Lee for providing an opportunity to give a talk on decoupling in his group seminar. The author is also thankful  to Dr. Abhishek Ghosh for some discussions on decoupling theory,  and 
      Zhuoran Li for bringing into the author's knowledge that the main result of this article is closely related to a more general result of \cite{Zhuoran} after the submission of this article in arXiv.
        Some part of the work was done when the author was a  Post Doctoral fellow at the Seoul National University and this work has been partially supported by NRF grant no. 2022R1A4A1018904 funded by the Korea government(MSIT). This work is also supported by DST inspire faculty award\\
        (ref. no. DST/INSPIRE/04/2023/002187) and Indian Institute of Science, Bengaluru. 

\bibliographystyle{plain}

\end{document}